\documentclass[10pt,oneside]{amsart}

\usepackage{amssymb}
\usepackage{formel}
\usepackage{diagram}

\theoremstyle{plain}
\newtheorem{theorem}{Theorem}[section]
\newtheorem{prop}[theorem]{Proposition}
\newtheorem{lemma}[theorem]{Lemma}
\newtheorem{corollary}[theorem]{Corollary}
\newtheorem{definition}[theorem]{Definition}
      
\theoremstyle{remark}
\newtheorem{remark}[theorem]{Remark}
\newtheorem{example}[theorem]{Example}

\numberwithin{equation}{section}

\makeatletter
\def\@setcopyright{}
\def\serieslogo@{}
\makeatother

\begin{document}


\author{Tobias Pecher}
\address{Emmy-Noether-Zentrum, Department Mathematik}
\address{Universit\"at Erlangen-N\"urnberg, Bismarckstrasse 1 $\!\!\frac{1}{2}$, 91054 Erlangen, Germany}
\email{pecher@mi.uni-erlangen.de}

\title[Multiplicity-free Super vector spaces]{Multiplicity-free Super vector spaces}

\begin{abstract}
Let $V$ be a complex finite dimensional super vector space with an action of a connected semisimple group $G$. We classify those pairs $(G,V)$ for which all homogeneous components of the super symmetric algebra of $V$ decompose multiplicity-free.
\end{abstract}




\date{\today}

\maketitle

\section{Introduction} \label{sec:intro}


One of the main questions of invariant theory is, given a Lie group $G$ with a representation $V$, the determination of a set of generators of the invariant algebra $\CC[V]^G$ as well as finding all relations among them. Solutions to these questions are usually called a First and Second Fundamental Theorem of invariant theory of the underlying representation.
In the most common examples, $G$ is a classical group and $V$ a number of copies of the defining representation (and maybe some copies of its dual representation). An important role in classical invariant theory was played by the famous identity of Capelli~\cite{Cap}

\begin{equation} \label{eqn:capelli}
\det[\Delta_{ij} + \delta_{ij}(n - j)] = \det(x_{ij}) \cdot \Omega.
\end{equation}
It describes a relation between the so-called ``polarization operators'' $\Delta_{ij}$ on the variables $x_{ij}$ and the $\Omega$-operator which was introduced by A. Cayley~\cite{Cay} and had also applications to invariant theoretical problems. 
Some applications of~\eqref{eqn:capelli}, such as a FFT for $G = \OG_m$, can be found in H. Weyl's book~\cite{We}.
\bigskip

In~\cite{HoX, Ho1, Ho2}, R. Howe developed invariant theory on the basis of multiplicity-free actions. More precisely, he could reprove well-known instances of FFT's and SFT's by using certain multiplicity-free actions, i.e. representation spaces $V$ for which $\CC[V]$ is multiplicity-free. Moreover, he showed that for a multiplicity-free action $(G,V)$, the algebra of $G$-invariant differential operators $\PD(V)^G$ is a polynomial algebra and that such pair naturally gives to an identity as~\eqref{eqn:capelli}. Namely, given the canonical homomorphism

\begin{equation} \label{eqn:}
\mathcal{Z}(\Lg) \longrightarrow \mathbb{PD}(V)^G
\end{equation}
from the center of the enveloping algebra of the Lie algebra $\mathfrak{g}$ of $G$ into the $G$-invariant differential operators, one could ask for a preimage of such operator under this homomorphism. The classical example~\eqref{eqn:capelli} comes from the action of $\GL_n\times\GL_n$ on the space of $n \times n$-matrices~\cite{Ho2}. \bigskip

In~\cite{Ho1} one also finds FFT's and SFT's for exterior invariant algebras $(\bigwedge V)^G$. It turns out that this 
``skew invariant theory'' is, in some sense, very similar to its symmetric counterpart and some examples can be explained via multiplicity-free exterior algebras (or, as called by Howe, {\it skew multiplicity-free modules}).

In this article, we study study the super symmetric setting and therefore combine the two types of multiplicity-free actions above. Hence, let $V$ be a (complex) {\it super vector space}, i.e. a finite dimensional complex vector space together with a fixed direct sum decomposition (or a $\ZZ_2$ grading) $V = V_0 \boxplus V_1$ of subspaces $V_0$ and $V_1$. They are referred to as the {\it even} and {\it odd} part of $V$. If $V_1 = 0$ (resp. $V_0 = 0$) we call $V$ {\it purely even} (resp. {\it purely odd}). In contrast, by a {\it proper super space} we mean that $V_0$ and $V_1$ are nonzero. A representation of $G$ on a super vector space is given by the direct sum of representations on $V_0$ and $V_1$. \bigskip

Our main object of interest will be the supersymmetric algebra $P(V)$ on $V$. This can be defined as

\begin{equation} \label{eqn:supersymm-algebra}
P(V) = S(V_0) \otimes \bigwedge(V_1)
\end{equation}
the symmetric algebra on the even part of $V$ tensored with the exterior algebra on the odd part. We assume that $V$ is a module for a connected semisimple group $G$. Since both parts of a generic $V$ can be reducible under $G$, we use the $\boxplus$ notation to indicate the splitting. Thus, let

\begin{equation} \label{eqn:grading}
V_0 = U_1 \oplus \dots \oplus U_k,~ V_1 = W_1 \oplus \dots \oplus W_l
\end{equation}
be the decomposition of even and odd part into irreducible submodules of $G$. By basic multilinear algebra, the super symmetric algebra $P(V) = S(V_0)\otimes \bigwedge(V_1)$ decomposes as a direct sum into $G$ stable subspaces

\begin{equation} \label{eqn:multilin}
P(V) = \bigoplus_{({\bf i}, {\bf j})} S^{i_1} U_1 \otimes \dots \otimes S^{i_k} U_k \otimes \bigwedge^{j_1} W_1 \otimes \dots\otimes \bigwedge^{j_l} W_l.
\end{equation}

Let us denote the direct summands on the right hand side of~\eqref{eqn:multilin} by $P^{({\bf i}, {\bf j})}(V)$. 

\begin{definition} \label{def:superMF}
The pair $(G,V)$ is called super multiplicity-free (or super MF), if 

\[ \dim\Hom_G\left(P^{({\bf i}, {\bf j})}(V), \Gamma \right) \leq 1 \]

for all appropriate multiindices $({\bf i}, {\bf j})$ and all irreducible representations $\Gamma$ of $G$.
\end{definition}

Assume that $(G,V)$ is super MF according to Definition~\ref{def:superMF} with $V$ decomposing as in~\eqref{eqn:grading}. Consider the bigger group $\tilde{G} = (\CC^*)^{k+l} \times G$ and extend the action on $V$ as follows: Let the $i$th copy of $\CC^*$ act trivially on all irreducible submodules except on $U_i$ (resp. $V_{i -k}$ if $k < i \leq k + l$) where it should act by multiplication with scalars. (We call this particular action a {\it saturated action}.) Then it can be stated that $(G,V)$ is super MF (according to Definition~\ref{def:superMF}) if and only if the super symmetric algebra $P(V)$ is multiplicity-free as a $\tilde{G}$ module, i.e.

\begin{equation} \label{eqn:MF}
P(V) = \bigoplus_{\lambda \in \Xi} V(\lambda)~{\rm with}~ V(\lambda) \neq V(\mu) ~{\rm if}~ \lambda \neq \mu.
\end{equation}

The problem is that $P(V)$ might or might not be multiplicity-free under a smaller group $S \times G$, where $S$ is a subgroup of the $(\CC^*)^{k+l}$ torus. Since this is not easy to deal with, we shall classify only those representation that satisfy Definition~\ref{def:superMF} or, equivalently, multiplicity-free super symmetric algebras for saturated actions. \bigskip

At this point we shall make a further remark about our interest in super MF actions. First, let $V$ be a super vector space. Write $\tilde{V} = V \oplus V^*$ and consider the canonical super symmetric pairing of $V$ and $V^*$ given by

\begin{equation} \label{eqn:canoncial_pairing}
\left< v + \lambda, w + \mu \right> := \lambda(w) - (-1)^{|v||w|} \mu(v), 
\end{equation}

where $v, w \in V$ and $\lambda, \mu \in V^*$ are homogeneous with respect to the $\ZZ_2$ grading on $V$, resp. $V^*$. In the free associative algebra over $V$ with $1$, define the super commutator of two elements $x,y \in V$ by

\begin{equation} \label{eqn:super-commutator}
[ x,y ] := x\cdot y - (-1)^{|x||y|} \cdot y\cdot x.
\end{equation}

Then, we can consider the quantized algebra $P(\tilde{V})_{\left< \cdot,\cdot\right>}$ subject to the relations

\begin{equation}
[x,y] =  \left<x,y \right> \cdot 1
\end{equation}
for all $x,y \in V$. This algebra 
can be realized as a endomorphism algebra of $P(V)$. Furthermore, given an action of any {\it reductive} group $G$ on $V$, the following is known~\cite[Theorem 3]{Ho2}:

\begin{prop} \label{prop:quantum-iso}

There is a natural $G$-equivariant isomorphism $\varphi: P(\tilde{V})_{\left< \cdot,\cdot\right>} \rightarrow \PD(V)$. \hfill $\Box$ \end{prop}

\begin{remark} \label{rem:harmonic}
Now assume that $P(V)$ is multiplicity-free under $G$ as in~\eqref{eqn:MF} (althoguh here, $G$ is not necessarily of the form). Then, by Proposition~\ref{prop:quantum-iso}, the subalgebra of $\PD(V)$ of $G$-invariant differential operators has the form

\[ \PD(V)^G = \bigoplus_{\lambda \in \Xi} (V(\lambda)\otimes V(\lambda)^*)^G. \]

By Schur's Lemma we find a canonical basis $D_\lambda$ for these operators. Applying this lemma yet another time, we find by the $G$ invariance and by~\eqref{eqn:MF}, each $D_\lambda$ maps on each irreducible subspace $V(\mu)$ as a scalar $c_\lambda(\mu)$. It is desirable to determine both the basis and the spectral values of these elements for every super MF space $(G,V)$. 
\end{remark}

This paper is organized follows: In Section~\ref{sec:prelim} we introduce the notion of the representation diagram of a module $V$ for a semisimple group $G$ and state the classification of purely even and odd super MF spaces. Moreover, we recall decomposition formulas for symmetric and exterior algebras of some particular representations of linear groups. They will be used in some calculations in the preceeding sections.

As our main result, the classification of saturated super multiplicity-free actions (Theorem~\ref{thm:proper-super}) is given in Section~\ref{sec:results}. Beside this we show two interesting properties of super MF spaces: First, we prove in Proposition~\ref{prop:subgraph} that for a given super MF representation $(G,V)$ every smaller representation (which is given by a subgraph of the representation diagram of $(G,V)$) is also super MF. This fact reduces the number of necessary calculations tremendously. But also the second result is labor-saving: It characterizes a super MF space $(G,V)$ by the fact that its algebra of invariant differential operators is abelian (Corollary~\ref{cor:MF-abelian}). With this, we deduce that the property of a module $V$ to be super MF does not depend upon possible exchanges of some submodules by their duals (Corollary~\ref{cor:dual-statements}).

The rest of this article is concerned with establishing Theorem~\ref{thm:proper-super}: In Section~\ref{sec:decomp}
we show that all modules that are listed in this theorem (except for some modules of simple groups) are indeed super MF.
The remaining ones are treated in Section~\ref{sec:simple}, which is mainly devoted to the completeness of part a) of Theorem~\ref{thm:proper-super}. Here, we make use the decomposition formulas that are stated in Section~\ref{sec:prelim}. (However, in some preliminary calculations we also used the computer algebra packages LiE~\cite{LCL} and Schur~\cite{Wy}.) In contrast, the considerations in Section~\ref{sec:two} and Section~\ref{sec:three} that deal with the completeness of part b) and c) of Theorem~\ref{thm:proper-super} mostly rely on the statement on subdiagrams in Proposition~\ref{prop:subgraph}. \bigskip

{\it Acknowledgements:} This work is part of the author's PhD research conducted at University of Erlangen-N\"urnberg under the supervision of Friedrich Knop. The author would like to thank him for his support and suggestions. 





\section{Preliminaries} \label{sec:prelim}

In the following, we will classify all pairs $(G,V)$ satisfying condition Definition~\ref{def:superMF}. The two extremal cases of purely even and odd super MF spaces have been investigated already. The even case $V = V_0$ has been treated by ~\cite{Ka}, \cite{BR1}, \cite{Lea}, while the odd case $V = V_1$ was covered by \cite{Ho1} and \cite{Pe}. Hence, we are left with the problem to find all instances of proper super spaces that are super MF. By definition of the super symmetric algebra \[ P(V) = S(V_0) \otimes \bigwedge(V_1) \]
it is immediate that both, even and odd part of $V$ have to be super MF according to Definition~\ref{def:superMF}. Therefore, we recall the clasification results for these two special cases. But before, we introduce the notion a representation diagram. This is a convenient tool for visualizing the action of $G$ on $V$. \bigskip

Let $G = G_1 \times \dots \times G_h$ a group with $h$ simple factors and $V = U_1 \oplus \dots \oplus U_k \boxplus W_1 \oplus \dots \oplus W_l$ be a decomposition into irreducible subspaces. Draw three horizontal lines of dots, one upon the other, where the top row consists of $k$ dots labelled by the $U_p$, the row in the middle consists of $h$ dots,labelled by the $G_r$ and the bottom row consists of $l$ dots labelled by the $W_q$. These $p + q + r$ dots form the vertices of the graph. For each pair of a simple group factor $G_r$ and an irreducible submodule we draw an edge between the corresponding vertices if and only if $G_r$ acts nontrivially on the submodule. 

\begin{definition} \label{def:rep-graph}
The graph $\mathcal{G} = \mathcal{G}(G,V)$ obtained by the above procedure is called the representation graph (or representation diagram) of $(G,V)$.
\end{definition}

\begin{example} \label{ex:graph}
 The representation graph of the action of $\SL_n \times \SL_2 \times \SO_{2m+1}$ on $\CC^n\otimes\CC^2 \boxplus \CC^2 \otimes \CC^{2m+1}$ is given by

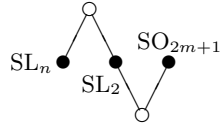
\[ \superN{\SL_n}{\SL_2}{\SO_{2m+1}}{}{}. \]
\end{example}

Now we can state the classification of super MF spaces for $V = V_0$ and $V = V_1$.

\begin{theorem}[Symmetric case] \label{thm:symmetric} 
Let $(G,V)$ be an indecomposable representation with $V$ being a purely even super space. 

{\upshape a)} If $V$ is irreducible, then all instances of super MF spaces are given by the following list:

\[\begin{array}{|l | l | l | l | l |} \hline
\SL_n  ~ (n\geq 2)          & \Sp_{2n} ~ (n \geq 2)   & \Delta_7    & G_2  & \SL_n\otimes\SL_m  ~ (n,m \geq 2) \\
S^2\SL_n ~ (n\geq 2)        & \SO_{2n+1} ~ (n \geq 2) & \Delta_9    & E_7  & \SL_k\otimes\Sp_{2n} ~ (k=2,3)    \\
\bigwedge^2\SL_n~(n \geq 4) & \SO_{2n} ~ (n \geq 4)   & \Delta_{10} &      & \SL_n\otimes\Sp_4 ~ (n \geq 2)    \\ \hline
\end{array}\]

{\upshape b)} If $V$ is reducible but indecomposable, then all instances of super MF spaces are given by:

$\symmV{\SL_n}{}{}$,                \hfill $\symmV{\SL_n}{}{\bigwedge^2\CC^n}\,\,\,\,\,\,$, \hfill 
$\symmV{\Sp_{2n}}{}{}$,             \hfill $\symmV{\Spin_8}{}{\Delta_8^+}$,                 \hfill
$\symmN{\SL_n}{\SL_m}{}{}\!\!\!\!$,    

$\symmN{\SL_2}{\Sp_{2m}}{}{}$, \hfill 
$\symmM{\SL_n}{\SL_2}{\SL_m}{}{}$, \hfill
$\symmM{\SL_n}{\SL_2}{\Sp_{2m}}{}{}$, \hfill
$\symmM{\Sp_{2n}}{\SL_2}{\Sp_{2m}}{}{}$\!\!\!\!\!.
\end{theorem}

\begin{proof}
See~\cite[Theorem 3]{Ka} for the irreducible case and~\cite[Theorem 2.5]{Lea} or~\cite[Theorem 2]{BR1} for the reducible.  
\end{proof}

\begin{theorem}[Skew-symmetric case] \label{thm:skew-symmetric}

Let $(G,V)$ be an indecomposable representation with $V$ being a purely odd super space.

{\upshape a)}  If $V$ is irreducible and $G$ simple, then all instances of super MF spaces are given by the following list:

\[ \begin{array}{|l|l|l|} \hline
\SL_n~(n \geq 2)           & \SO_{2n+1}~(n\geq 2) & \Sp_{2n}~(n\geq 3) \\
S^2\SL_n~(n\geq 2)         & \Delta_7             & \bigwedge^2_0\Sp_4 \\
S^k\SL_2~(k=3,\dots,6)     & \Delta_9             & \bigwedge^3_0\Sp_6 \\
S^3\SL_3                   & \SO_{2n} ~(n \geq 3) & G_2                \\
\bigwedge^2\SL_n~(n\geq 4) & \Delta_{10}^+        & E_6                \\
\bigwedge^3\SL_6           & \Delta_{12}^+        & E_7                \\ \hline
\end{array} \]

{\upshape b)} If $V$ is as in a), but $G$ not necessarily simple then:

$\SL_n\otimes \SL_m$, $\SL_n\otimes \Sp_4$, $\SL_k\otimes \SO_{2m+1}~(k=2,3)$ or $\SL_2\otimes \SO_{2m}$. \bigskip

{\upshape c)} If $V$ is reducible but indecomposable, then all instances of super MF spaces are given by:

$\skewA{\SL_n}{}{}$,                 \hfill
$\skewA{\SL_k}{}{\!\!\!\!\!\!\!S^2\CC^k}$, \hfill
$\skewA{\SL_2}{}{\!\!\!\!S^l\CC^2}$,  \hfill
$\skewA{\SO_{2n+1}}{}{}$,  \hfill
$\skewN{\SL_n}{\SL_m}{}{}\!\!\!\!$,  \hfill    
$\skewN{\SL_2}{\SL_n}{\!\!\!S^2\CC^2}{}\!\!\!\!$,  \hfill

$\skewN{\!\!\SL_2}{\!\!\SO_{2n+1}}{}{}$\!\!, 
$\skewN{\!\!\SL_2}{\!\!\SO_{2n+1}}{\!\!\!S^2\CC^2}{}$,  \hfill
$\skewW{\SL_n}{\!\!\!\SL_2}{\SL_m}{}{}$\!\!\!\!,  \hfill
$\skewW{\SL_n}{\!\!\!\!\SL_2}{\!\!\!\!\SO_{2m\!+1}}{}{}$\!,  \hfill
$\skewW{\!\!\!\SO_{2n\!+1}}{\!\!\!\SL_2}{\!\!\!\!\SO_{2m\!+1}}{}{}$.
\end{theorem}

\begin{proof}
Part a) is proven in~\cite[Theorem 4.7.1]{Ho1}, parts b) and c) in~\cite[Theorem 4.8]{Pe}.
\end{proof}

A very important instance of a multiplicity-free space is given by the action of $\GL_n \times \GL_m$ on $V = \CC^n \otimes \CC^m$. This action is irreducible and we can intepret it either as a purely even or purely odd super space. Accordingly, we are dealing with the symmetric or exterior algebra on $V$. In both cases, the arising decomposition is not only multiplicity-free but also the isotypic components for $\GL_n$ and $\GL_m$ in this decomposition stand in a bijective correspondence. This is called {\it $(\GL_n,\GL_m)$ (skew) duality} and the explicit formulas for the homogeneous components are given by

\begin{equation} \label{eqn:GLnGLm_duality}
S^k(\CC^n\otimes \CC^m) = \bigoplus_{|\lambda|=k, \ell(\lambda)\leq \min\{n,m\}} V(\lambda)^{(n)} \otimes V(\lambda)^{(m)}
\end{equation}

and

\begin{equation} \label{eqn:GLnGLm_skew_duality}
\bigwedge^k(\CC^n\otimes \CC^m) = \bigoplus_{|\lambda|=k, \ell(\lambda)\leq n, \lambda_1 \leq m} V(\lambda)^{(n)} \otimes V(\lambda^t)^{(m)}.
\end{equation}

A proof for both cases can be found in~\cite{Ho1}. We will use these actions and the bijective correspondence for further calculations.

There are two other series of modules for which symmetric and antisymmetric plethysms are easy to describe, namely the actions of $\GL_n$ on $S^2\CC^n$ and $\bigwedge^2\CC^n$. For this reason it is convenient to introduce the following terms: A Young diagram of shape $(s,1^t)$ is called a {\it $(s,t)$-hook}. A sequence of $(s_i,t_i)$-hooks, where $s_{i+1} \leq s_{i} - 1$ and $t_{i+1} \leq t_{i} - 1$ can be combined to a {\it nested hook}, i.e. a regular Young diagram $\lambda$ such that the boxes in the angle of each hook form the diagonal of $\lambda$. We have

\begin{eqnarray}
S^n(S^2\CC^n)                &=& \bigoplus_{|\lambda|=2n,~\lambda_i~even} V(\lambda),     \label{eqn:S2-sym} \\
S^n(\bigwedge^2\CC^n)        &=& \bigoplus_{|\lambda|=2n,~\lambda_i~even} V({\lambda^t}) \label{eqn:L2-sym} 
\end{eqnarray}

for the symmetric powers of $S^2\CC^n$ and $\bigwedge^2\CC^n$, while their skew symmetric powers decompose by

\begin{eqnarray}
\bigwedge^n(S^2\CC^n)        &=& \bigoplus_{\lambda \in D} V(\lambda),                         \label{eqn:S2-skew} \\
\bigwedge^n(\bigwedge^2\CC^n) &=& \bigoplus_{\lambda \in E} V(\lambda).                        \label{eqn:L2-skew}
\end{eqnarray}

Here, $D$ is the set of all partitions that consist of a nested $(r_i+1,r_i-1)$-hook with $\sum r_i = n$; and $E$ the set of all partitions that consist of a nested $(r_i,r_i)$-hook where also $\sum r_i = n$. (Proofs for all these decomposition can be found in~\cite{Ho1, Ho2}.) \bigskip

We have no generalization of~\eqref{eqn:S2-skew} to $V = S^k\CC^n$ for arbitray $k$. But at least for the second exterior power one can easily show that

\begin{equation} \label{eqn:Sk-plethysms}
\bigwedge^2 S^k\SL_2 = \bigoplus_{j = 0,\dots, \lfloor \frac{k-1}{2} \rfloor} V(2k - 2 - 4j)
\end{equation}

by writing down its character polynomial. Since $V = V_1 = S^k\SL_2$ is super MF, we need the above formula in some computations in Section~\ref{sec:simple}.

\section{Results} \label{sec:results}


Let $\rho: G \rightarrow V$ be a representation of a semisimple group on a super vector space $V$. The super MF property of $(G,V)$ actually only depends on $\rho(G)$. Hence, if $\varphi: H \rightarrow G$ is a surjective homomorphism, then $(H,V)$ is also super MF. Furthermore, it is obvious that $(G,V)$ is super MF if and only if $(G,V^*)$ has this property. 

\begin{definition} \label{def:geom_equiv}
Two representations $(G,\rho,V)$ and $(G',\rho',V')$ are said to be {\itshape geometrically equivalent} if there exists an isomorphism $\psi:V \stackrel{\sim}{\longrightarrow} V'$ such that for the induced isomorphism  $\GL(\psi): \GL(V) \rightarrow \GL(V')$ one has $\GL(\psi)(\rho(G))=\rho'(G')$. 

We write  $(G,\rho,V) \sim (G',\rho',V')$ for a pair of geometrically equivalent representations or, if the underlying homomorphisms are obvious, $(G,V)\sim (G',V')$.
\end{definition}

This definition takes into account the two problems stated above: It is immediate that $(G,\rho,V)$ and $(H,\rho\circ\varphi,V)$ are geometrically equivalent. Moreover, if we fix a maximal torus $T$ of $G$, then there is a automorphism $\theta$ of $G$ such that $\theta(t) = t^{-1}$ for all $t \in T$. By this, also $(G,V)$ and $(G,V^*)$ are geometrically equivalent.

Note that there are infinitely many series of reducible super MF spaces since for every such pair $(G_i,V_i)$ we can build the direct sum 

\begin{equation} \label{eqn:decomposable}
(G_1 \times G_2, V_1 \oplus V_2)
\end{equation}
which is again super MF since $P(V_1 \oplus V_2) = P(V_1) \otimes P(V_2)$. In order to avoid such a situation, we a representation {\it decomposable} if it is geometrically equivalent to a representation as in~\eqref{eqn:decomposable}.  Otherwise, we call it {\it indecomposable} and hence we will classifiy the indecomposable super MF ones. From the Definition~\ref{def:rep-graph} of a representation diagram it follows immediately that indecomposable representations are in $1 - 1$ correspondence with {\it connected} representation graphs. \bigskip

Every subgraph $\mathcal{G}'$ of a representation diagram $\mathcal{G}$ is also a representation diagram. A very useful result is the following:

\begin{prop} \label{prop:subgraph}
If the representation described by $\mathcal{G}$ is super MF, then the same is true for the underlying representation of $\mathcal{G}'$.
\end{prop}

\begin{proof} This is proven in~\cite[Proposition 4.7]{Pe} for purely odd super spaces. The proof makes use of $(\GL_n,\GL_m)$ skew duality~\eqref{eqn:GLnGLm_skew_duality} and can be translated word-by-word to the general setting by simply exchanging skew duality~\eqref{eqn:GLnGLm_skew_duality} by the symmetric counterpart~\eqref{eqn:GLnGLm_duality}, if necessary. \end{proof}

As we shall see, the representation that is desribed by the graph in Example~\ref{ex:graph} is super MF. By  Proposition~\ref{prop:subgraph} it follows that e.g. \[ \!\!\!\superIV{\SL_2}{\SO_{2n+1}}{}{} \,\,\, \]  is also super MF. \bigskip

Before we state our main result, we give a characterization of super MF spaces in terms of the $G$-invariant differential operators $\PD(V)^G$. Denote by $\CC[G]$ the group algebra of a reductive group $G$.

\begin{theorem} \label{lemm:End-irreducibly}
Let $G$ be a reductive group that acts on a super vector space $V$. Then, as a joint $\CC[G]\otimes \PD(V)^G$ module,

\begin{equation} \label{eqn:super-duality}
P(V) = \bigoplus_{\lambda \in \Xi} V(\lambda)\otimes V^\lambda, 
\end{equation}
where $V^\lambda$ are pairwise non-isomorphic ireducible represntations of $\PD(V)^G$. \hfill $\Box$
\end{theorem}

This theorem is a super symmetric version of~\cite[Theorem 4.5.14]{GW} and is based on a generalization of Burnside's theorem. It has the following important


%
%

\begin{corollary} \label{cor:MF-abelian}
Let $G$ be a semisimple group acting on $V = V_0 \boxplus V_1$. This representation is super MF if and only if $\PD(V)^{\tilde G}$ is abelian.
\end{corollary}

\begin{proof}
Let $(G,V)$ be super MF, i.e. $P(V)$ is multiplicity-free under $\tilde{G}$. Then, the commutativity of $\PD(V)^{\tilde{G}}$ follows as in Remark~\ref{rem:harmonic}. On the other hand, if $\PD(V)^{\tilde{G}}$ is abelian, all representations $V^\lambda$ in~\eqref{eqn:super-duality} have dimension $1$.
\end{proof}

Although being an alternative characterization of super MF modules, this criterion is not easily read off from an arbitrary module $V$, since the task of determining the algebra structure of $\PD(V)^G$ is at least as hard as determining the multiplicities of the $G$ module $P(V)$. However, it makes it possible to deduce a labor-saving property on dual submodules: 
Recall that $(G,V) \sim (G, V^*)$ for arbitrary representations. Unfortunately, it is not true in general that $(G, V \oplus W) \sim (G,V \oplus W^*)$ for $G$ modules $V$ and $W$. Nevertheless, we will claim that a space $(G,V\oplus W)$ is super MF whenever $(G,V\oplus W^*)$ is super MF. Thus, in the classification of super MF spaces, we can ignore all extra cases coming from replacing a submodule by its dual.

\begin{corollary} \label{cor:dual-statements}
Let $G$ be a group that acts on super vector spaces $V = V_0 \boxplus V_1$ and $W = W_0 \boxplus W_1$. Then, $V \oplus W$ is super MF if and only if $V \oplus W^*$ is super MF.
\end{corollary}

\begin{proof}
By Proposition~\ref{prop:quantum-iso}, $\PD(V\oplus W) \simeq P(\tilde{V} \oplus \tilde{W})_{\left<\cdot,\cdot\right>} \simeq \PD(V \oplus W^*)$. Since this is also an algebra isomorphism, $\PD(V\oplus W)^G$ is commutative if and only if $\PD(V \oplus W^*)^G$ is commutative and the claim follows from Corollary~\ref{cor:MF-abelian}.
\end{proof}

With this result by hand, we can state the classification in a compact form.

\begin{theorem} \label{thm:proper-super}
Let $(G,V)$ be an indecomposable super MF representation on a proper super space $V$. Then, $G$ has at most three simple factors. Up to geometric equivalence (and a possible exchange of irreducible submodules by their duals), all super MF are given by the following list. \bigskip

{\upshape a)} If $G$ is simple:

1) $\superI{\SL_n}{}{}$,                   \hfill 2) $\superI{\SL_2}{}{S^2\CC^2}$, \hfill
3) $\superI{\SL_4}{}{\bigwedge^2\CC^4}$,   \hfill 4) $\superI{\SL_n}{S^2\CC^n}{}$, \hfill 
5) $\superI{\Sp_4}{}{\bigwedge^2_0\CC^4}$, \bigskip

{\upshape b)} If $G = G_1 \times G_2$:

1) $\!\!\!\superIV{\SL_n}{\SL_m}{}{}$,  \hfill     
2) $\!\!\!\superIV{\SL_2}{\SO_{2n+1}}{}{} \,\,\,\,\,\,\,$, \hfill
3) $\!\!\!\superAI{\SL_n}{\SL_m}{}{}$,  \hfill    
 
4) $\!\!\!\superAI{\SL_2}{\SL_n}{}{S^2\CC^{2}}$, \hfill
5) $\!\!\!\superAI{\SL_2}{\Sp_{2n}}{}{}$,   \hfill
6) $\!\!\!\superAI{\SL_2}{\Sp_{2n}}{}{S^2\CC^{2}}$ \bigskip

{\upshape c)} If $G =  G_1 \times G_2 \times G_3$:

1) $\!\!\!\!\!\superN{\SL_n}{\SL_2}{\SL_m}{}{}\!\!\!\!$,
2) $\!\!\superN{\Sp_{2n}}{\SL_2}{\SL_m}{}{}\!\!\!\!$, \hfill
3) $\!\!\superN{\SL_n}{\SL_2}{\SO_{2m+1}}{}{}$,   \hfill
4) $\!\!\!\!\!\superN{\Sp_{2n}}{\SL_2}{\SO_{2m+1}}{}{}$\,\,\,\,.

\end{theorem}

\begin{corollary}
If $(G,V)$ is a saturated super MF representation then it is either occurs either in Theorem~\ref{thm:symmetric} or~\ref{thm:skew-symmetric} (if $V$ is purely even or odd) or, if $V$ is a proper super space, it occurs in Theorem~\ref{thm:proper-super}.
\end{corollary}

\section{Decompositions of super MF modules} \label{sec:decomp}

%

In this section we verify the super MF condition of all modules of the form $V = V_0 \boxplus V_1$ with $V_0$ and $V_1$ nontrivial that are listed in Theorem~\ref{thm:proper-super}b) and c). By the result on subgraphs in Proposition~\ref{prop:subgraph} this also shows that some modules of Theorem~\ref{thm:proper-super}a) are super MF. We begin with a lemma that limits the possible cases of super MF modules.

\begin{lemma} \label{lemm:big-mama}
$\SL_n\otimes \SL_p \boxplus \SL_p \otimes \SL_m$ is super MF if and only if $p = 2$.
\end{lemma}

\begin{proof} First, let $p = 2$. There is an integer $a$, depending on $l$ and $m$ such that the decomposition of $P^{(k|l)}$ is given by

\[ \left(\bigoplus_{i = 0,\dots, \lfloor \frac{k}{2} \rfloor} V(k - i, i)\otimes V(k - 2i)\otimes \CC \right) \otimes \left( \bigoplus_{j = a,\dots, \lfloor \frac{l}{2} \rfloor} \CC \otimes V(l - 2j)\otimes V(2^j, 1^{l - 2j}) \right). \]

The tensor products of the occuring representations for $\SL_2$ are multiplicity-free and the representations for $\SL_n$ and $\SL_m$ are pairwise different. Hence, the supersymmetric algebra is MF for $(\CC^*)^2\times G$. Note that this case refers to representation 1) of Theorem~\ref{thm:proper-super}c). \bigskip

If $p \geq 3$, $P^{(3|3)} = (V(2,1)\otimes V(2,1)\otimes\CC \oplus \dots ) \otimes (\CC\otimes V(2,1)\otimes V(2,1) \oplus \dots)$ is not multiplicity-free, since $V(2,1)\otimes V(2,1) = 2 \cdot V(3,2,1) \oplus \dots$ is not.
\end{proof}

In the following, it is convenient to use a different notation. For the irreducible representations of $\GL_n$ or $\SL_n$ we will sometimes write $\{\lambda\}$ instead of $V(\lambda)$. Classically, $\{\lambda\}$ refers to the character of $V(\lambda)$ if $\ell(\lambda) \leq n$. But this symbol even makes sense for $\ell(\lambda) > n$ if we set $\{\lambda\} = 0$ in this case. Similarly, we use the notations $[\lambda]$ and $\left<\lambda \right>$ for the groups $\SO_m$ (with $m = 2n$ or $2n+1$) and $\Sp_{2n}$. Also here, these symbols are well-defined even for $\ell(\lambda) \geq n$. But here, the modification rules are nontrivial and $[\lambda]$ (resp. $\left<\lambda \right>$) can refer to a virtual representation. For a rigorous definition of these {\it universal characters} see~\cite{KT}. \bigskip

One advantage of the universal characters is that the branching rules $\SL_m\downarrow \OG_m$ and $\SL_{2n}\downarrow \Sp_{2n}$ for irreducible representations can be formulated in a simple way. The corresponding formulas are well-known (and proven e.g. in~\cite{Li}).

\begin{prop}[Littlewood] \label{prop:univ-branching} For any partition $\lambda$ the following identities hold:

\begin{itemize}
\item[\rm a)] $\{ \lambda \} = \sum_{\mu \subseteq \lambda} \left( \sum_{\beta~even} c^\lambda_{\mu,\beta} \right) [\mu]$,

\item[\rm b)] $\{ \lambda \} = \sum_{\mu \subseteq \lambda} \left( \sum_{\beta^t~\!\!even} c^\lambda_{\mu,\beta} \right) \left< \mu \right>$.
\end{itemize} 

The numbers $c^\lambda_{\mu,\beta}$ are the Littlewood-Richardson coefficients. \hfill $\Box$
\end{prop}

Two particular cases are of special interest: If $\lambda = (2^a, 1^b)$, we have for the restriction $\SL_m\downarrow\OG_m$ 

\begin{equation} \label{eqn:so-beispiel}
\{\lambda\} = \sum_{i = 0}^a [2^{a-i}, 1^{b-i}],
\end{equation}

while for $\SL_{2n}\downarrow \Sp_{2n}$ and $\lambda^t = (k, l)$ with $k \geq l$,

\begin{equation} \label{eqn:sp-beispiel}
\{\lambda\} = \sum_{i = 0}^l \left< k - i, l - i  \right>.
\end{equation}

Given an arbitrary $\lambda$, if we want to know which irreducible representations for $\OG_m$ and $\Sp_{2n}$ actually occur in the restriction of $V(\lambda)$, we must apply the modification rules to the terms on the right hand side of Proposition~\ref{prop:univ-branching} a) and b). However, in the special case~\eqref{eqn:sp-beispiel} for $\Sp_{2n}$ this is not necessary, because $\ell(k-i,l-i) \leq 2$ is inside the ``stable range''. For $\OG_m$ the modification rules for the special case~\eqref{eqn:so-beispiel} take the following form: Let $\mu = (2^c,1^d)$. If $\ell(\mu) > n$ then, depending on the parameter $h  = 2\ell(\mu) - m$, the modification rule has the following form

\begin{equation} \label{eqn:special-mod}
[\mu] = \optidr{[2^c,1^{d-h}]}{h\leq d}{[2^{c - (h - d - 1)},1^{h-d-2}]}{h>d+1}{0}{h=d+1}.
\end{equation}

The modification rules in the general case are described in~\cite{Ki1} (or either in~\cite{Pe}).

\begin{lemma} \label{lemm:proof_1}
Representations 2) $\Sp_{2n}\otimes \SL_2 \boxplus \SL_2\otimes\SL_m$, 3) $\SL_n\otimes \SL_2 \boxplus \SL_2\otimes\SO_{2m+1}$ and 4) $\Sp_{2n}\otimes \SL_2 \boxplus \SL_2\otimes\SO_{2m+1}$ of Theorem~\ref{thm:proper-super}c) are super MF. In particular, also Representations 2) $\SL_2 \boxplus \SL_2\otimes\SO_{2m+1}$ and 5) $\Sp_{2n}\otimes\SL_2 \boxplus \SL_2$ of Theorem~\ref{thm:proper-super}b) are super MF.
\end{lemma}

\begin{proof} We show that all the $\Sp_{2n}$ (resp. $\SO_{2m+1}$) representations in the decomposition of $S^k(\Sp_{2n}\otimes\SL_2)$, resp. $\bigwedge^l(\SL_2\otimes\SO_{2m+1}$) are pairwise different. This suffices to prove the claim for the representations 2), 3) and 4) of Theorem~\ref{thm:proper-super}. First, by the skew duality ~\eqref{eqn:GLnGLm_skew_duality} and the branching rule for $\Sp_{2n}$ we have for any $k$

\begin{equation} \label{eqn:sp_good-decomposition}
S^k(\Sp_{2n}\otimes \SL_2) = \bigoplus_{i = 0,\dots, \lfloor \frac{k}{2} \rfloor} \left( \bigoplus_{j = 0,\dots, i} \left< k  - i - j, i - j \right> \right)\otimes \{ k - 2i \}
\end{equation}

and so the claim is obvious. For $\SO_{2m+1}$ the statement is not immediate since we have to deal with modification rules in that case. Let $0\leq l \leq 2(2n + 1)$, so

\begin{equation} \label{eqn:so_good-decomposition}
\bigwedge^l \SL_2\otimes \SO_{2n+1} = \bigoplus_{\lambda\in\mathcal{P}_l} \{\lambda \} \otimes \{ \lambda^t \},
\end{equation}

where $\mathcal{P}_l = \{ \lambda: |\lambda|=l, \ell(\lambda)\leq 2, \lambda_1\leq 2n + 1 \} = \{ (a, l - a), (a - 1, l - a + 1), \dots, (b, l - b) \}$ for some integers $a \geq b \geq 0$. Thus, we have to show that $M_l := \bigoplus_{\lambda\in\mathcal{P}_l} {\rm Res}^{\SL_{2n+1}}_{\SO_{2n+1}}(V_\lambda)$ is multiplicity-free. If we choose an integer $a \geq s \geq b$, we have for $\lambda = (s, l - s) \in \mathcal{P}_l$ by the branching rule for universal characters

\[ \{ \lambda^t \} = [(s, l - s)^t] + [s( - 1, l - s - 1)^t] + \dots + [(2s - l + 1, 1)^t] + [(2s - l, 0)^t]. \] 

This results in a multiplicity-free decomposition

\[ M_l = \sum_{s = a,\dots,b} \sum_{\nu = 0, \dots, l - s} [s - \nu, l - s - \nu] =: \sum_{\tau\in I}[\tau]. \]

In order to analyze the impact of the modification rules we re-order this sum by setting $\mathcal{S}_q = \{ \tau\in I: {\rm ~second~column~of~\tau~has~exactly~} q {\rm ~boxes} \}$, so that
\begin{equation} 
\sum_{\tau\in I}[\tau] = \sum_{q\in\NN_0} \sum_{\tau\in\mathcal{S}_q} [\tau].
\end{equation}

It suffices to consider only partitions $\tau \in I$ with $\ell(\tau) > n$ which need at most one strip removal and that affects only the first column. First, note that for every $q\in\NN_0$ there exist $p,r\in\NN_0$ such that

\[ \mathcal{S}_q = \{ (p+q,q)^t, (p+q+2,q)^t, \dots, (p+q+2r,q)^t \}.\] 

Since all the numbers $\{ \ell(\lambda) : \lambda\in\mathcal{S}_q \}$ are pairwise different and differ by at least $2$, it cannot happen that two different elements in $\mathcal{S}_q$ become equal after modification. Observe further that the numbers $\{ \lambda^t_1 - \lambda^t_2: \lambda \in \mathcal{P}_i \}$ are either all even or all odd. So, modified elements from $\mathcal{S}_q$ cannot equal any unmodified elements  in $\mathcal{S}_q$, too.
\end{proof}

\begin{lemma} \label{lemm:proof_3}
Representations 1) $\SL_n \boxplus \SL_n \otimes \SL_m$ and 3) $\SL_n\otimes \SL_m \boxplus \SL_m$ of Theorem~\ref{thm:proper-super}b) are super MF. In particular, representation 5) of Theorem~\ref{thm:proper-super}a) is super MF.
\end{lemma}

\begin{proof} By the Pieri rule, all direct summands on the right hand side of 

\[ P^{(k|l)}(\SL_n \boxplus \SL_n \otimes \SL_m)  = \bigoplus_{|\lambda| = l,~ \ell(\lambda) \leq n,~ \lambda_1 \leq m} V(k)^{(n)} \otimes V(\lambda)^{(n)} \otimes V(\lambda^t)^{(m)} \] 

are multiplicity-free. They can be distinguished by the representations of the $\SL_m$ factor, so {\it 1)} is super MF. 
The same argument applies to {\it 3)}. \end{proof}

\begin{lemma} \label{lemm:proof_4}
Representations 4) $\SL_2 \otimes \SL_n \boxplus S^2\SL_2$ and 6) $\SL_2\otimes \Sp_{2n} \boxplus S^2\SL_2$ of Theorem~\ref{thm:proper-super}b) are super MF. In particular, representation 2) $SL_2\boxplus S^2\SL_2$ of Theorem~\ref{thm:proper-super}a) is super MF.
\end{lemma}

\begin{proof} The exterior powers of $S^2\SL_2$ are either trivial, or isomorphic to $V(2)$. Hence $P(\SL_2 \otimes \SL_n \boxplus S^2\SL_2)$ is multiplicity-free. The same is true for $P(\SL_2 \otimes \Sp_{2n} \boxplus S^2\SL_2)$ by~\eqref{eqn:sp_good-decomposition}. \end{proof}

\section{Representations of simple groups} \label{sec:simple}


In this section, $G$ is a simple group. We show that Theorem~\ref{thm:proper-super}a) contains all super MF representations of such $G$. In particular, we have to verify the super MF condition for representations {\it 3)} - {\it 5)} of this theorem. Since their representation diagrams do not occur as subdiagrams of bigger super MF representations, this has not been done yet. \bigskip

So let $V = V_0 \boxplus V_1$ be a proper super space, with $V_0$ coming from Theorem~\ref{thm:symmetric}a) and $V_1$ from Theorem~\ref{thm:skew-symmetric}a). For each root type of $G$ there is given one table with the different choices for $V_0$ along the rows and for $V_1$ along the columns. (A $\bigstar$ indicates that the ranks of the groups that act on $V_0$ and $V_1$ cannot be equal.) This also yields the proof that representations {\it 3)} - {\it 5)} of Theorem~\ref{thm:proper-super}a) are super MF. In type $A$ we often use formulas~\eqref{eqn:S2-sym} - \eqref{eqn:L2-skew}.

\[ \begin{array}{| l | c | c | c | c | c | c | c | c | c |} \hline
{\rm Type~} {\bf A}&\SL_n &S^2\SL_n &\bigwedge^2\SL_n &S^3\SL_2 & S^4\SL_2 & S^5\SL_2 & S^6\SL_2 & S^3\SL_3 & \bigwedge^3\SL_6 \\ \hline
\SL_n  &  1a) & 1b)     & 1c)             & 1d)     & 1e)      & 1f)      & 1g)      & 1h)      & 1i)              \\ \hline
S^2\SL_n& 2a) & 2b)     & 2c)             & 2d)     & 2e)      & 2f)      & 2g)      & 2h)      & 2i)              \\ \hline
\bigwedge^2\SL_n & 3a) & 3b)  & 3c)       & \bigstar& \bigstar & \bigstar & \bigstar & \bigstar & 3i)              \\ \hline
\end{array} \] \bigskip

\begin{itemize} \setlength{\itemsep}{8pt}

\item[1a)] This is representation {\it 1)} of Theorem~\ref{thm:proper-super}a).

\item[1b)] $n \geq 4$: $P^{(2|3)}(\SL_n\boxplus S^2\SL_n) = \{2\}\cdot (\{4,1^2\} + \{3^2\}) = 2\cdot \{4,3,1\} + \dots$ 

\item[1c)] $n \geq 5$: $P^{(2|5)} = \{2\}\cdot (\{4,2^2,1^2\} + \{3^2,2^2\} + \dots ) = 2 \cdot \{4,3,2^2,1\} + \dots$ \bigskip

\noindent $n = 4$: $\bigwedge^l(\bigwedge^2\SL_4)$ is irreducible, except for $l = 3$. And in this case, $P^{(k|3)}$ decomposes as

\[ \{k\}\!\cdot\!(\{2\}\!+\!\{2^3\}) = \{k\!+\!2\} + \{k\!+\!1, 1\} + \{k, 2\} + \{k\!+\!2, 2^2\} + \{k, 1^2\} + \{k\!-\!2\}. \]

\noindent In particular, representation {\it 3)} of Theorem~\ref{thm:proper-super}a) is super MF.

\item[1d) - 1g)] It follows from~\eqref{eqn:Sk-plethysms} that for $p = 3, \dots, 6$ we have

\noindent $P^{(2|2)}(\SL_2 \boxplus S^p \SL_2) = \{2\}\cdot (\sum_{k=0}^{\lfloor\frac{p}{2}\rfloor} \{2p - 4k\}) = 2 \cdot \{2p - 2\} + \dots$

\item[1h)] $P^{(2|2)}(\SL_3 \boxplus S^3\SL_3) = \{2\} \cdot (\{3^2\} + \{5,1\}) = 2\cdot \{5,3\} + \dots$

\item[1i)] $P^{(2|5)}(\CC^6 \boxplus \bigwedge^3 \SL_6) = \{2\} \cdot (\{1^3\} + \{5,3^2,2^2\} + \dots) = 2 \cdot \{3,1^2\} + \dots$

\item[2a)] It follows from~\eqref{eqn:S2-sym} that for $k \in \NN$ and $0 \leq l \leq n$

\[ P^{(k|l)}(S^2\SL_n\boxplus\SL_n) = \left( \bigoplus_{|\lambda|=2k,~\!\lambda_i~\!even} V(\lambda) \right)\otimes V(1^l). \]

In particular, $\lambda_i \equiv \lambda'_i~{\rm mod}~2$ for all occuring partitions $\lambda$, $\lambda'$. By the Pieri-Rule (see e.g. \cite[p. 388]{GW}), the decomposition of $V(\lambda)\otimes V(1^l)$ is a direct sum of $V(\mu)$ where $\mu$ is obtained from $\lambda$ by adding $l$ boxes, but no two in the same row. So, $V(\lambda)\otimes V(1^l)$ and $V(\lambda')\otimes V(1^l)$ share no isomorphic submodule if $\lambda \neq \lambda'$. (Even after considering modification rule for rows with $n$ boxes.) Thus, representation {\it 4)} of Theorem~\ref{thm:proper-super}a) is super MF.

\item[2b)] $n\geq 2$: $P^{(2|1)}(S^2\SL_n \boxplus S^2\SL_n) = (\{2^2\} + \{4\}) \cdot \{2\} = 2 \cdot \{4,2\} + \dots$

\item[2c)] $n \geq 4$: $P^{(2|2)}(S^2\SL_n \boxplus \bigwedge^2\SL_n) = (\{2^2\} + \{4\}) \cdot \{2,1^2\} = 2 \cdot \{4,2,1^2\} + \dots$.

\item[2d) - 2i)] $P^{(1|k)}(S^2\SL_n \boxplus V) = P^{(2|k)}(\SL_n \boxplus V)$. See 1d) - 1i). 

\item[3a)] $n \geq 4$: $P^{(3|2)}(\bigwedge^2\SL_n \boxplus \SL_n) = (\{3^2\} + \{2^2,1^2\} + \dots)\cdot \{1^2\} = 2 \cdot \{ 3^2, 1^2 \} + \dots$

\item[3b)] $n \geq 3$: $P^{(5|3)}(\bigwedge^2 \SL_n \boxplus S^2\SL_n)= \{5^2\} \cdot (\{4,1^2\} + \{3^2\}) = 2\cdot \{8,6,2\}+ \dots$

\item[3c)] $n\geq 4$: $P^{(6|3)} = (\{4^2,2^2\} + \dots)\cdot(\{3,1^3\} + \{2^3\} + \dots) = 2\cdot\{6,5,4,3\} + \dots$

\item[3i)] $P^{(1|3)}(\bigwedge^2 \SL_6 \boxplus \bigwedge^3 \SL_6) = \{1^2\}\cdot (\{1^3\}+ \{3^2,1^3\} + \dots) = 2 \cdot \{2^2,1\} + \dots$

\end{itemize} \bigskip

\[ \begin{array}{| l | c | c | c |}  \hline
{\rm Type~} {\bf C}& \Sp_{2n} & \bigwedge^2_0\Sp_4 & \bigwedge^3_0\Sp_6 \\ \hline
\Sp_{2n}           & 1a)      & 1b)                & 1c)                \\ \hline
\bigwedge^2_0\Sp_4 & 2a)      & 2b)                & \bigstar           \\ \hline
\end{array} \]

\begin{itemize} \setlength{\itemsep}{8pt}

\item[1a)] $P^{(2|2)} = \left< 2 \right> \cdot \left( \left< 0 \right> + \left< 1^2\right>\right) = 2 \cdot \left<2\right> + \left<2,1^2 \right> + \left<1,1\right> + \left<3,1\right>$ 

\item[1b)] $P^{(k|1)} = P^{(k|4)} = \left<k\right> \cdot \left<1^2\right> = \left<k+1,1\right> + \left<k-1,1 \right> + \left<k \right>$, and
$P^{(k|2)} = P^{(k|3)} = \left<k+2\right> + \left<k+1,1\right> + \left<k\right>  +  \left<k,2\right> + \left<k-1,1\right> + \left<k-2\right>$.

\noindent In particular, representation {\it 5)} of Theorem~\ref{thm:proper-super}a) is super MF.

\item[1c)] $P^{(2|3)} = \left<2\right>\cdot \left(\left<1^3\right> + \left<3,2\right>\right) = 2\cdot \left<3,2\right>+\dots$

\item[2a)] $P^{(2|2)} = \left(\left<0\right> + \left<2^2\right>\right) \cdot \left(\left<1^2\right> + \left<0\right>\right) = 2 \cdot \left<1^2\right> + \dots$ 

\item[2b)] $P^{(2|1)} = \left(\left<0\right> + \left<2^2\right> \right) \cdot \left<1^2\right> = 2\cdot\left<1^2\right> + \left<3^2\right> + \left<3,1\right>$ 

\end{itemize}

\bigskip

For the remaining types of groups $B$, $D$, $E$ and $G$ we shall mainly label the highest weights relative to the basis of fundamental weights of the corresponding group. 

\[ \begin{array}{| l | c | c | c |}                        \hline
{\rm Type~} {\bf B} & \SO_{2n+1} & \Delta_7 & \Delta_9  \\ \hline
\SO_{2n+1}          & 1a)        & 1b)      & 1c)       \\ \hline
\Delta_7            & 2a)        & 2b)      & \bigstar  \\ \hline
\Delta_9            & 3a)        & \bigstar & 2c)       \\ \hline
\end{array} \]

\begin{itemize} \setlength{\itemsep}{8pt}

\item[1a)] $P^{(2|2)} = ([0] + [2])\cdot [1^2] = 2\cdot[1^2] + [2,1^2] + [2] + [3,1]$ 

\item[1b)] $P^{(2|3)} = ((2,0,0) + (0,0,0)) \cdot ((0,0,1) + (1,0,1)) = 2\cdot (1,0,1) + \dots$ 

\item[1c)] $P^{(2|3)} = ((2,0,0,0) + (0,0,0,0)) \cdot ((0,1,0,1) + (1,0,0,1)) = 2\cdot (1,0,0,1) +\dots$

\item[2a)] $P^{(3|2)} = ((0,0,1) + (0,0,3))\cdot (0,1,0) = 2\cdot(0,1,1) + \dots$

\item[2b)] $P^{(2|3)} = ((0,0,2) + \dots)\cdot ((1,0,1) + \dots) = 2\cdot (1,0,1) + \dots$ 

\item[3a)] $P^{(3|2)} = ((0,0,0,1) + (0,0,0,3))\cdot (0,1,0,0) = 2\cdot (0,1,0,1) + \dots$ 

\item[3c)] $P^{(2|3)} = ((0,0,0,2) + \dots) \cdot ((1,0,0,1) + \dots) = 2\cdot (1,0,0,1) + \dots$

\end{itemize}

\bigskip

For the half spin representations we use the convention $\Delta_{2n}^+ = V_{\omega_{n-1}}$ and $\Delta_{2n}^- = V_{\omega_n}$. Note that $(\Delta_{10}^+)^* = \Delta_{10}^-$ while $\Delta_8^{\pm}$ and $\Delta_{12}^{\pm}$ are self-dual.

\[ \begin{array}{| c | c | c | c | c | c |}                                                       \hline
{\rm Type~}{\bf D} & \SO_{2n}  & \Delta_8^+ & \Delta_{10}^{+} & \Delta_{12}^+  & \Delta_{12}^- \\ \hline
\SO_{2n}           & 1a)       & 1b)        & 1c)             & 1d)            & 1e)           \\ \hline
\Delta_8^+         & 2a)       & 2b)        & \bigstar        & \bigstar       & \bigstar      \\ \hline
\Delta_{10}^{+}    & 3a)       & \bigstar   & 3b)             & \bigstar       & \bigstar      \\ \hline
\end{array} \]

\begin{itemize} \setlength{\itemsep}{8pt}

\item[1a)] is not super MF, since it has the same decomposition as B 1a).

\item[1b)] $P^{(3|4)} = ((1,0,0,0) + (3,0,0,0)) \cdot ((2,0,0,0) + \dots) = 2 \cdot\!(1,0,0,0) + \dots$

\item[1c)] $P^{(2|5)} = ((2,0,0,0,0) + \dots)\cdot ((1,1,0,0,1) + \dots) = 2\cdot (1,1,0,0,1) + \dots$

\item[1d)] $P^{(2|5)} = ((2,0,0,0,0,0) + \dots) \!\cdot\! ((0,1,1,0,0,1) + \dots) = 2 \cdot\! (0,1,1,0,0,1) + \dots$ 

\item[1e)] $P^{(2|5)} = ((2,0,0,0,0,0) + \dots) \!\cdot\! ((0,1,1,0,1,0) + \dots) = 2 \cdot\! (0,1,1,0,1,0) + \dots$ 

\item[2a)] $P^{(3|4)} = ((0,0,1,0) + (0,0,3,0)) \cdot ((0,0,2,0) + \dots) = 2\cdot (0,0,1,0) + \dots$

\item[2b)] $\Delta_8^+ \boxplus \Delta_8^+$ is geometrically equivalent to $\CC^8 \boxplus \CC^8$.

\item[3a)] $P^{(3|3)} = ((1,0,0,1,0) + \dots) \cdot((0,0,1,0,0) + \dots) = 2 \cdot (0,1,0,1,0) + \dots$ 

\item[3b)] $P^{(3|2)} = ((1,0,0,1,0) + \dots) \cdot (0,0,1,0,0) = 2 \cdot (1,0,0,0,1) + \dots$

\end{itemize}

The only exceptional groups that admit both symmetric MF {\it and} skew MF representations are $G_2$ and $E_6$. But $G_2 \boxplus G_2$ and $E_6 \boxplus E_6$ are not super MF, since 

\[ \begin{array}{c c l}
P^{(1|2)}(G_2 \boxplus G_2) &=& (1,0) \cdot ((1,0) + (0,1)) = 2\cdot (1,0) + \dots   \\
                            & &                                                      \\
P^{(2|1)}(E_6 \boxplus E_6) &=& ((0,0,0,0,0,1) + (2,0,0,0,0,0))\cdot (1,0,0,0,0,0)   \\    
                            &=& 2\cdot(1,0,0,0,0,1) + \dots
\end{array} \]

Let $(G,V)$ be a saturated indecomposable super MF representation. Now we know that, if $V$ has exactly two irreducible submodules, then there are no other equivalence classes than those in Theorem~\ref{thm:proper-super}a). The next lemma (in connection with Proposition~\ref{prop:subgraph}) shows that $V$ cannot have more than two irreducible submodules. In particular, this concludes the proof of Theorem~\ref{thm:proper-super}a).

\begin{lemma} \label{lemm:ex-three}
A saturated indecomposable representation $(G, V)$ with three irreducible submodules is never super MF.
\end{lemma}

\begin{proof} Let $V = V_1 \oplus V_2 \oplus V_3$ with an arbitrary arrangement of these submodules to even and odd part of $V$. We claim that multiplicities occur already in $P^{(1,1,1)}(V_1\oplus V_2 \oplus V_3) = V_1 \otimes V_2 \otimes V_3$. This follows from the fact that in type $A$, for any integers $k \geq l \geq 1$, the tensor product of $\{1\}\cdot\{l\}\cdot\{k\}$ decomposes by

\begin{equation} \label{eqn:three_tensor}
\{k+l+1\} + 2\cdot \left(\sum_{i=0}^{l-1} \{k + l -i,i+1\}  \right) + \delta_{l<k}\cdot \{k,l+1\} + \left( \sum_{i=1}^{l} \{k + l -i, i ,1\} \right)
\end{equation}

and thus, is not multiplicity-free. (Here, $\delta_{l<k}$ equals $1$ if $l<k$, and $0$ otherwise. In the case that $G$ is isomorphic to $\SL_2$ the rightmost sum vanishes.) 

\bigskip

Assume for the moment that $V = V_1 \oplus V_2 \oplus V_3$ is purely odd. Since every pair of irreducible submodules must be skew MF, we can deduce that either $G$ is isomorphic to $\SL_n$, $V_1 \simeq V_2 \simeq \CC^n$ and $V_3 \simeq S^k\CC^n$, or $G = \SO_{2n+1}$ with $V_i =\CC^{2n+1}$. But $V$ fails to be skew MF in the former case by~\eqref{eqn:three_tensor}, and by a branching argument also in the latter case. \\

With a similar reasoning, this can also be proved for proper super spaces $V = V_1 \boxplus V_2 \oplus V_3$ and $V = V_1 \oplus V_2 \boxplus V_3$. 
\end{proof}

\section{Representations of groups with two simple factors} \label{sec:two}

Now we turn to the case where $(G,V)$ is a saturated indecomposable representation of a semisimple group $G$ with simple factors $G_1$, $G_2$. For the moment, let $V$ be a proper super spaces with $V_0$ and $V_1$ being irreducible. The next lemma tells us that there can be no super MF spaces related to diagrams of the form 
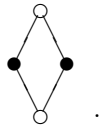
\[ \LOOP. \] 

\begin{lemma} \label{lemm:two_2} If $(G, V)$ is super MF, then it is impossible that both simple factors of $G$ act nontrivially on $V_0$ and $V_1$. \end{lemma}

\begin{proof}
We consider the pair $(G,V) = (G_1 \times G_2, V_0 \boxplus V_1)$ with $V_i = W^{(1)}_i \otimes W^{(2)}_i$. For the moment we assume that $(G,V) = (\SL_n\times \SL_m, \CC^n\otimes \CC^m \boxplus \CC^n\otimes \CC^m)$. Then, the homogeneous components of $P(V)$ are given by

\begin{equation} \label{eqn:local}
P^{(k|l)} = \bigoplus_{\begin{minipage}{90pt}\tiny $|\lambda|=k,~\ell(\lambda)\leq min\{n,m\} \\ |\mu| = l,~ \ell(\mu) \leq n,~ \mu_1\leq m$ \end{minipage}} \left( V_\lambda^n \otimes V_\mu^n \right) \otimes \left( V_\lambda^m \otimes V_{\mu^t}^m \right).
\end{equation}

If $n,m \geq 3$ consider $k = l = 3$ and $\lambda = \mu = (2,1)$. Then, $V^n_{(3,2,1)}$ has multiplicity two inside $V^n_{(2,1)} \otimes V^n_{(2,1)}$. If $n = 2$ consider $k = l = 2$. So, we get terms that are labelled by $\lambda, \mu \in \{ (2), (1^2) \}$ and both, $V^2(2) \otimes V^2(2)$ and $V^2(1^2) \otimes V^2(1^2) \simeq \CC$ contain the trivial representation $\CC$. Now, if $G_i$ are arbitrary, we get plethysms

\begin{equation} \label{eqn:plethysm}
G_1 \times G_2 \rightarrow \GL(W^{(1)}_i) \times \GL(W^{(2)}_i) \rightarrow \GL(W^{(1)}_i \otimes W^{(2)}_i). 
\end{equation}

Therefore, the multiplicities of the $G_1 \times G_2$ module $P(V_0 \boxplus V_1)$ are governed by~\eqref{eqn:local} and hence they are always $\geq 2$.
\end{proof}

We keep the above assumptions on $V$ and assume that $(G,V)$ is saturated indecomposable super MF. By Lemma~\ref{lemm:two_2}, there is one part $X$ of the super vector space on which both simple factors of $G$ act nontrivially, while on the other part $Y$ there is exactly one simple factor acting nontrivially. The respective representation diagrams look like \[ \superIV{}{}{}{} {\rm ~ and~} \!\!\!\superAI{}{}{}{}. \]

Let us consider the case $X = V_1$ and $Y = V_0$ first. Below we list all combinations $V = V_0 \boxplus V_1$ with $V_0$ being a representation of a simple group from Theorem~\ref{thm:symmetric}a) and $V_1 = W^{(1)}_1 \otimes W^{(2)}_1$ being a representation of a group with two simple factors from Theorem~\ref{thm:skew-symmetric}b), such that $V$ is not contained in Theorem~\ref{thm:proper-super}b). \bigskip

Note that by Proposition~\ref{prop:subgraph}, we can avoid calculations for all $V$ such that $V_0 \boxplus W^{(i)}_1$ is not super MF. By this reasoning, we can omit computations for the following modules:

$\bigwedge^2\SL_n \boxplus \SL_n \otimes \SL_m~(n \geq 4, m \geq 2)$, 
$S^2\SL_2 \boxplus \SL_k\otimes S^2\SL_2~(k = 2,3)$, 
$\bigwedge^2\SL_n \boxplus \SL_n \otimes \Sp_4~(n\geq 4)$, 
$\Sp_4 \boxplus \SL_n \otimes \Sp_4$,
$\bigwedge^2_0 \Sp_4 \boxplus \SL_n \otimes \Sp_4$,
$\SO_{2n} \boxplus \SL_2\otimes \SO_{2n}$, 
$\Delta_{8}^+ \boxplus \SL_2 \otimes \SO_8$,
$\Delta_{10}^+ \boxplus \SL_2\otimes \SO_{10}$,
$\bigwedge^2 \SL_3 \boxplus \SL_3 \otimes \SO_{2n+1}$,
$\SO_{2n+1} \boxplus \SL_k \otimes \SO_{2n+1}$, 
$\Delta_7 \boxplus \SL_k \otimes \SO_7$ and 
$\Delta_9 \boxplus \SL_k \otimes \SO_9$. In all these cases, $k$ equals $2$ or $3$. \bigskip

For the remaining modules, we compute in each case a graded subspace $P^{(i|j)}(V)$ of $P(V)$ that contains multiplicities.

\[ \begin{array}{|l l l|} \hline
{\rm Representation }                                 & (i|j)     & {\rm Multiplicity}                     \\ \hline
S^2\SL_n\boxplus \SL_n\otimes \SL_m                   & (5|2)     & 2 \cdot \{8,4\}\otimes\{1^2\}          \\ 
\SL_2 \boxplus \SL_k \otimes S^2\SL_2~(k = 2,3)       & (2|3)     & 2 \cdot \{3-k\}\otimes\{4\}            \\
\SL_4 \boxplus \SL_2 \otimes \bigwedge^2\SL_4         & (1|4)     & 2 \cdot \{2\}\otimes\{2^2,1\}          \\
\SL_n \boxplus \SL_n \otimes \Sp_4                    & (1|3)     & 2 \cdot \{3,1\}\otimes \left<1\right>  \\
\SL_2 \boxplus \SL_2\otimes \SO_{2n}                  & (2|n+1)   & 2 \cdot \{n+1\}\otimes[1^{n-1}]        \\
\SL_3 \boxplus \SL_3 \otimes \SO_{2n+1}               & (1|3)     & 2 \cdot \{1\}\otimes[1]                \\   
\Delta_5 \boxplus \SL_3\otimes \SO_5                  & (2|3)     & 2 \cdot \{0\}\otimes (1,2)             \\ 
\Delta_5 \boxplus \SL_2\otimes \SO_5                  & (2|3)     & 2 \cdot \{1\}\otimes (0,2)             \\ \hline
\end{array}
\]

Since the first representation in this list is not super MF, one can show that also $S^2\SL_k \boxplus \SL_k\otimes \SO_m$ and $S^2\SL_n \boxplus \SL_n\otimes \Sp_4$ are not super MF. This follows from an argument similar to that in~\eqref{eqn:plethysm}. \bigskip

Now we interchange the role of even and odd part, i.e. $X = V_0$ and $Y = V_1$. In this case we only have to consider the groups $\SL_n \times \SL_m$ and $\SL_n \times \Sp_4$.

\[ \begin{array}{|l l l|} \hline
{\rm Representation }                                                 & (i|j) & {\rm Multiplicity}                  \\ \hline
\SL_n\otimes \SL_m \boxplus \bigwedge^2\SL_m~ (n \geq 2,~ m \geq 4)   & (3|2) & 2\cdot\{2,1\}\otimes\{3,2,1^2\}     \\ 
\SL_n \otimes \Sp_{2m} \boxplus \SL_n~ (n>3,m = 2~ {\rm or} ~n=3,m>2) & (3|1) & 2\cdot\{2,1\}\otimes\left<1\right>  \\ 
\SL_n\otimes \Sp_4 \boxplus  \bigwedge^2_0 \Sp_4~ (n \geq 2)          & (3|2) & 2\cdot\{2,1\}\otimes\left<1^2\right>\\ \hline
\end{array}
\]

Here, we could omit calculations for the modules
$\SL_n \otimes \SL_2 \boxplus S^p\SL_2~(p = 3,\dots,6)$,
$\SL_m \otimes \SL_n \boxplus S^2\SL_m~(m \geq 3, n \geq 2)$,
$\SL_n \otimes \SL_3 \boxplus S^3\SL_3$,
$\SL_n \otimes \SL_6 \boxplus \bigwedge^3\SL_6$,
$\SL_n \otimes \Sp_4 \boxplus S^2\SL_n$ $(n \geq 3)$,
$\SL_m \otimes \Sp_4 \boxplus \bigwedge^2\SL_m$,
$\SL_n\otimes \Sp_{2m} \boxplus \Sp_{2m}$ $(n > 3, m = 2$ or $n \leq 3, m > 2)$ and
$\SL_k\otimes \Sp_6 \boxplus \bigwedge^3_0 \Sp_6$ $(k = 2,3)$.

\bigskip

Now we turn to super vector spaces that decompose into three irreducible submodules for $G$ so in particular, either $V_0$ or $V_1$ is irreducible. The possible representation diagrams are given in the lemma below. We show that $(G,V)$ is never super MF in these cases. 


\begin{lemma} \label{lemm:SL-diagrams}
Representations of the following type are never super MF:

i) $\superIIV{}{}{}{}{}$ \hfill ii) $\superAII{}{}{}{}{}$ \hfill iii) $\superVV{}{}{}{}{}$ \hfill iv) $\superAA{}{}{}{}$

v) $\superIN{}{}{}{}{}$ \hfill vi) $\superNI{}{}{}{}$ \hfill    vii) $\superIIN{}{}{}{}{}$ \hfill vii) $\superNII{}{}{}{}{}$ 
\end{lemma}

\begin{proof} Assume for the moment that $G = \SL_n \times \SL_m$. Let $U = S^k$ and $W = S^l$ be the defining representations (or its symmetric square) of $\SL_n$ or $\SL_m$ and denote by $V = U \oplus W \oplus \CC^n\otimes \CC^m$ a super vector space without deciding yet to which part the direct summands belong. Observe that for fixed choice of $U$ and $W$ the decomposition of 

\[ P^1 U \otimes P^1 W \otimes P^2 (\CC^n\otimes\CC^m) \]

depends only on the parity of $\CC^n \otimes \CC^m$ in $V$, but not on those of $U$ and $W$. For example, fix $U = S^k\CC^n$, $W = S^l\CC^m$ (with $k,l = 1,2$), then it follows that representations of type 1a) and 1b) are not super MF since
$P^{(1,1|2)}(U \oplus W \boxplus \CC^n\otimes \CC^m)$ (resp. $P^{(1|1,2)}(U \boxplus W \oplus \CC^n\otimes \CC^m)$) is given by

\[ \{k\}\otimes \{l\} \cdot (\{2\}\otimes\{1^2\} + \{1^2\}\otimes\{2\}) \]
and contains $\{k+1,1\}\otimes\{l+1,1\}$ with multiplicity two. The remaining cases can be ruled out by two completely analogous calculations. \bigskip

It follows by a branching argument that also for any other group $G$ the diagrams {\it i)} - {\it vii)} do not give rise for a super MF representation. \end{proof}


Suppose we are given an indecomposable saturated action $(G,V)$ with $G = G_1 \times G_2$ and $V$ consisting of three irreducible submodules or more. Its representation diagram has either a subdiagram which equals one of those in the above lemma, or any that describes the action of $G$ on a purely even (or purely odd) super vector space with three irreducible submodules. In both cases, $(G,V)$ fails to be super MF and thus the proof of Theorem~\ref{thm:proper-super}b) is complete. \bigskip

\section{Representations of groups more than two simple factors} \label{sec:three}

Finally, let $G = G_1 \times G_2 \times G_3$ be a group consisting of three simple factors. As above, we assume $(G,V)$ to be saturated indecomposable. Since the representation diagrams we are now dealing with are pretty big (they all consist of at least $4$ vertices) and since part a) and b) of Theorem~\ref{thm:proper-super} is now completely proven, we can heavily rely on Proposition~\ref{prop:subgraph} in the following.

\begin{lemma} \label{lemm:4.21}
The representations listed in Theorem~\ref{thm:proper-super}c) are the only super MF spaces of type
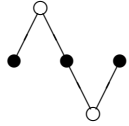
\[ \superN{}{}{}{}{}. \]
\end{lemma}

\begin{proof}
Let $V$ be a super MF space of the above type. Its representation diagram must contain subdiagrams corresponding to representations {\it 1)}, {\it 2)}, {\it 3)} or {\it 5)} in Theorem~\ref{thm:proper-super}b). By Lemma~\ref{lemm:big-mama}, 
the group factor that acts nontrivially on both parts of $V$ must be isomorphic to $\SL_2$. This forces $V$ to be equivalent to either representation {\it 1)} - {\it 4)} of Theorem~\ref{thm:proper-super}c). \end{proof}

In the case of purely odd super spaces, we get an instance of an irreducible super MF representation for a group with three simple factors. This is given by $\SL_2\otimes \SO_4$ and we show that this is the only example of a super MF space where more than two simple factors act nontrivially on a single irreducible sumand.

\begin{lemma} \label{lemm:4.23}
Let $\mathcal{G}$ be a connected representation diagram that contains 
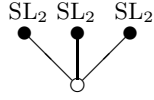
\[ \skewIII{\SL_2}{\SL_2}{\SL_2}{} \]
as a proper subgraph. Then the action described by $\mathcal{G}$ is never super MF.
\end{lemma}

\begin{proof} By assumption there is a subdiagram $\mathcal{G}'$ (with propably $\mathcal{G}' = \mathcal{G}$) that 
corresponds to either a purely odd super vector space $V = S^k\SL_2 \oplus \SL_2\otimes\SL_2\otimes \SL_2$ or to $V' = S^k\SL_2 \boxplus \SL_2\otimes\SL_2\otimes \SL_2$, where $1 \leq k \leq 4$. By restricting from $\GL_4$ to $\SO_4$ we can deduce from $(\GL_n,\GL_m)$ skew duality~\eqref{eqn:GLnGLm_skew_duality} the decompositon

\[ \bigwedge^3(\SL_2\otimes\SL_2\otimes\SL_2) = \{1\}\otimes\{1\}\otimes\{1\} + \{1\}\otimes\{1\}\otimes\{3\} + \{1\}\otimes\{3\}\otimes\{1\} + \{3\}\otimes\{1\}\otimes\{1\}. \]

It follows that in both cases there are multiplicities in $P^{(1|3)}(V')$, resp. $P^{|1,3)}(V)$ since

\[ P^1(S^k\SL_2) \otimes P^3(\SL_2^{\otimes3}) = \{k\}\!\otimes\!\{0\}\!\otimes\!\{0\} \cdot (\{1\}\!\otimes\!\{1\}\!\otimes\!\{1\} +\{3\}\!\otimes\!\{1\}\!\otimes\!\{1\} + \dots) \]

contains $\{k-1\}\!\otimes\!\{1 \}\!\otimes\!\{1\}$ with multiplicity $2$. \end{proof}

Now let $G = G_1 \times \dots \times G_k$ be a group with $k \geq 3$ simple factors and consider a saturated indecomposable representation $(G,V)$. We conclude that Theorem~\ref{thm:proper-super}c) is now proved since for the representation diagram of $(G,V)$ one of the following is true:

\begin{enumerate}

\item it describes the action of $G$ on a purely even (resp. odd) super vector space;

\item or it has a form as in Lemma~\ref{lemm:4.21} (if $k = 3$ and $V$ has exactly two irreducible submodules);

\item or it contains a connected subdiagram with two group vertices and three representation vertices (if $V$ has more than two irreducible submodules);

\item or there is an irreducible submodules on which at least $3$ simple factors act nontrivially. 

\end{enumerate}

In the first case, if $(G,V)$ is super MF, it is geometrically equivalent to one of the representations {\it 1)} - {\it 3)} from Theorem~\ref{thm:proper-super}c), while in the second case it is geometrically equivalent to one of {\it 4)} - {\it 10)}. In cases three and four $(G,V)$ can obviously not be super MF.

\end{document}